\documentclass[11pt]{article}

\usepackage{amsmath, amssymb, amsthm, calc}  

\title{Transference inequalities for multiplicative Diophantine exponents.
                              \thanks{ This research was  supported by
                              RFBR (grant $\textup N^{\circ}$ 09--01--00371a) and
                              the grant of the President of Russian Federation
                              $\textup N^\circ$ MK--1226.2010.1.
                             }}
\author{Oleg\,N.\,German}
\date{}

\theoremstyle{definition}
\newtheorem{definition}{Definition}

\theoremstyle{remark}
\newtheorem{remark}{Remark}

\theoremstyle{plain}
\newtheorem{theorem}{Theorem}
\newtheorem{lemma}{Lemma}

\newtheorem{corollary}{Corollary}

\newtheorem{classic}{Theorem}
\newtheorem{classicprime}[classic]{Theorem}

\DeclareMathOperator{\vol}{vol}

\DeclareMathOperator{\spanned}{span}

\renewcommand{\vec}[1]{\mathbf{#1}}
\renewcommand{\geq}{\geqslant}
\renewcommand{\leq}{\leqslant}
\renewcommand{\phi}{\varphi}

\newcommand{\R}{\mathbb{R}}
\newcommand{\Z}{\mathbb{Z}}

\newcommand{\cL}{\mathcal{L}}

\newcommand{\cS}{\mathcal{S}}
\newcommand{\cB}{\mathcal{B}}

\newcommand{\cH}{\mathcal{H}}

\newcommand{\tr}[1]{{#1}^\intercal}

\newcommand{\mbeta}{\beta_{\scriptscriptstyle M}}
\newcommand{\malpha}{\alpha_{\scriptscriptstyle M}}

\begin{document}

  \maketitle

  \begin{abstract}
    In this paper we prove inequalities for multiplicative analogues of Diophantine exponents, similar to the ones known in the classical case. Particularly, we show that a matrix is badly approximable if and only if its transpose is badly approximable and establish some inequalities connecting multiplicative exponents with ordinary ones.
  \end{abstract}

  \section{History and some objectives}

  Consider a system of linear equations
  \begin{equation} \label{eq:the_system}
    \Theta\vec x=\vec y
  \end{equation}
  with $\vec x\in\R^m$, $\vec y\in\R^n$ and
  \[ \Theta=
     \begin{pmatrix}
       \theta_{1,1} & \cdots & \theta_{1,m} \\
       \vdots & \ddots & \vdots \\
       \theta_{n,1} & \cdots & \theta_{n,m}
     \end{pmatrix},\qquad
     \theta_{i,j}\in\R. \]
  Let us denote by $\tr\Theta$ the transposed matrix and consider the corresponding ``transposed'' system
  \begin{equation} \label{eq:the_transposed_system}
    \tr\Theta\vec y=\vec x,
  \end{equation}
  where, as before, $\vec x\in\R^m$ and $\vec y\in\R^n$. Integer approximations to the solutions of
  the systems \eqref{eq:the_system} and \eqref{eq:the_transposed_system} are closely connected, which
  is reflected in a large variety of so called \emph{transference theorems}. Most of them deal with
  the corresponding asymptotics in terms of Diophantine exponents.

  \begin{definition} \label{def:beta}
    The supremum of real numbers $\gamma$, such that there are infinitely many $\vec x\in\Z^m$, $\vec
    y\in\Z^n$ satisfying the inequality
    \[ |\Theta\vec x-\vec y|_\infty\leq|\vec x|_\infty^{-\gamma}, \]
    where $|\cdot|_\infty$ denotes the sup-norm in the corresponding space,
    is called the \emph{(ordinary) Diophantine exponent} of $\Theta$ and is denoted by $\beta(\Theta)$.
  \end{definition}

  Considering a norm, other than the sup-norm, does not change much the nature of the phenomena
  observed, since all the norms on a Euclidean space are equivalent, and thus, such a change would
  not affect the exponents. However, substituting the sup-norm by a non-convex distance function
  seems to be a rather essential change. We deem the distance function generated by the geometric
  mean of coordinates to be the most interesting one in the class of non-convex distance functions,
  since it naturally leads us to Littlewood-like problems.

  For each $\vec z=(z_1,\ldots,z_k)\in\R^k$ let us define
  \[ \Pi(\vec z)=\left(\prod_{\begin{subarray}{c}1\leq i\leq k\end{subarray}}|z_i|\right)^{1/k}
     \quad\text{ and }\quad
     \Pi'(\vec z)=\left(\prod_{1\leq i\leq k}\max(1,|z_i|)\right)^{1/k}. \]

  \begin{definition} \label{def:mbeta}
    The supremum of real numbers $\gamma$, such that there are
    infinitely many $\vec x\in\Z^m\backslash\{\vec 0\}$, $\vec y\in\Z^n$ satisfying the inequality
    \begin{equation} \label{eq:mbeta}
      \Pi(\Theta\vec x-\vec y)\leq\Pi'(\vec x)^{-\gamma},
    \end{equation}
    is called the \emph{multiplicative Diophantine exponent} of $\Theta$ and is denoted by $\mbeta(\Theta)$.
  \end{definition}

  Ordinary and multiplicative exponents are connected by trivial inequalities
  \begin{equation} \label{eq:ord_and_mult_trivial}
    \beta(\Theta)\leq\mbeta(\Theta)\leq
    \begin{cases}
      m\beta(\Theta),\quad\text{ if }n=1, \\
      +\infty,\qquad\ \text{ otherwise},
    \end{cases}
  \end{equation}
  provided by the fact that for every $\vec z\in\R^k$ we have $\Pi(\vec z)\leq|\vec z|_\infty$, and
  for every $\vec z\in\Z^k$ we have $|\vec z|_\infty^{1/k}\leq\Pi'(\vec z)\leq|\vec z|_\infty$.

  On the other hand, Minkowski's convex body theorem gives us another pair of trivial bounds
  \begin{equation} \label{eq:by_minkowski}
    \mbeta(\Theta)\geq\beta(\Theta)\geq m/n,\quad\mbeta(\tr\Theta)\geq\beta(\tr\Theta)\geq n/m.
  \end{equation}

  As for any non-trivial relations on $\mbeta(\Theta)$ and $\mbeta(\tr\Theta)$, very little has been
  known so far. Schmidt and Wang \cite{schmidt_wang} proved in $1979$ that
  \begin{equation} \label{eq:schmidt_wang}
    \mbeta(\Theta)=m/n\iff\mbeta(\tr\Theta)=n/m,
  \end{equation}
  same as in the case of ordinary Diophantine exponents (see \cite{dyson}). Later, in $1981$, Wang and Yu \cite{wang_yu} proved that both equalities \eqref{eq:schmidt_wang} hold for almost all $\Theta$ with respect to the Lebesgue measure on $\R^{nm}$. This property is
  closely connected with the concept of \emph{badly approximable} matrices.

  \begin{definition} \label{def:BA}
    A matrix $\Theta$ is called \emph{badly approximable}, if
    \[ \inf_{\begin{subarray}{c} (\vec x,\vec y)\in\Z^m\oplus\Z^n \\ \vec x\neq\vec 0 \end{subarray}}
       |\vec x|_\infty^m|\Theta\vec x-\vec y|_\infty^n>0. \]
  \end{definition}

  \begin{definition} \label{def:MBA}
    A matrix $\Theta$ is called \emph{multiplicatively badly approximable}, if
    \[ \inf_{\begin{subarray}{c} (\vec x,\vec y)\in\Z^m\oplus\Z^n \\ \vec x\neq\vec 0 \end{subarray}}
       \Pi'(\vec x)^m\Pi(\Theta\vec x-\vec y)^n>0. \]
  \end{definition}

  It is well-known (see Theorem VIII in \cite{cassels}) that $\Theta$ is badly approximable if and only if $\tr\Theta$ is
  badly approximable. As for the multiplicative analogue of this property, the only fact known so far
  belongs to Cassels and Swinnerton-Dyer \cite{cassels_swinnerton}, who proved that if $n=2$, $m=1$,
  and $\Theta$ is multiplicatively badly approximable, then so is $\tr\Theta$. Notice that the
  existence of badly approximable $\Theta$ for $n=2$, $m=1$ is exactly the opposite to the statement
  of the Littlewood conjecture, so the case $n+m=3$ seems to be the most interesting one, however
  even in this case the implication has been known to hold only in one direction.

  As a corollary to the main theorem of the current paper (Theorem \ref{t:multitrans}) we get that
  $\Theta$ and $\tr\Theta$ are simultaneously badly approximable, for arbitrary $n$, $m$, thus
  filling this gap.

  Another relation connecting $\beta(\Theta)$ and $\beta(\tr\Theta)$ is Dyson's inequality
  \begin{equation} \label{eq:dyson}
    \beta(\tr\Theta)\geq\frac{n\beta(\Theta)+n-1}{(m-1)\beta(\Theta)+m}
  \end{equation}
  published by Dyson \cite{dyson} in $1947$ (it actually can be easily derived from Mahler's paper
  \cite{mahler_casopis_linear} of $1939$, see also \cite{mahler_matsbornik_dyson}), which generalizes
  Khintchine's famous transference principle formulated by Khintchine for the case when $n$ or $m$ is
  equal to $1$ (see \cite{khintchine_palermo}). As Bugeaud noticed in his paper \cite{bugeaud}, the
  argument used in \cite{schmidt_wang} allows to prove for $\Theta$ satisfying some additional
  assumptions that
  \begin{equation} \label{eq:schmidt_wang_bugeaud}
    \mbeta(\tr\Theta)\geq\frac{n\mbeta(\Theta)+n-1}{(m-1)\mbeta(\Theta)+m}\,.
  \end{equation}

  In the current paper we show (see Corollary \ref{cor:dysonlike}) that
  \eqref{eq:schmidt_wang_bugeaud} holds for all $n$, $m$ setting no restrictions on $\Theta$.

  The main theorem of this paper (Theorem \ref{t:multitrans}) is very similar to Mahler's theorem
  describing the transference principle in the case of ordinary Diophantine approximation (Theorem
  \ref{t:mahler_transference}, see also \cite{mahler_casopis_linear}, \cite{mahler_matsbornik_dyson},
  \cite{cassels}). Our theorem, however, has an unexpected feature, which mixes up a bit ordinary and
  multiplicative approximation and allows to obtain inequalities connecting ordinary and
  multiplicative exponents in the case when either $n$, or $m$ is equal to $1$.

  \section{Statement of the main theorem}

  One of the strongest theorems describing Khintchine's transference principle belongs to Mahler:

  \begin{theorem} \label{t:mahler_transference}
    If there are $\vec x\in\Z^m\backslash\{\vec 0\}$, $\vec y\in\Z^n$, such that
    \begin{equation} \label{eq:mahler_transference_hypothesis}
      |\vec x|_\infty\leq X,\qquad|\Theta\vec x-\vec y|_\infty\leq U,
    \end{equation}
    $0<U<1\leq X$, then there are $\vec x\in\Z^m$, $\vec y\in\Z^n\backslash\{\vec 0\}$, such that
    \begin{equation} \label{eq:mahler_transference_statement}
      |\vec y|_\infty\leq Y,\qquad|\tr\Theta\vec y-\vec x|_\infty\leq V,
    \end{equation}
    where
    \begin{equation} \label{eq:mahler_Y_V_in_terms_of_X_U}
      Y=(d-1)\big(X^mU^{1-m}\big)^{\frac{1}{d-1}},\quad
      V=(d-1)\big(X^{1-n}U^n\big)^{\frac{1}{d-1}},
    \end{equation}
    and $d=n+m$.
  \end{theorem}

  In \cite{german} a bit stronger version of Theorem \ref{t:mahler_transference} is proved. Namely,
  it appeared that the factor $d-1$ in \eqref{eq:mahler_Y_V_in_terms_of_X_U} can be substituted by a
  smaller factor tending to $1$ as $d\to\infty$. Let us describe this factor, since we shall use it
  in the statement of our main theorem.

  Denote by $\cB_\infty^d$ the unit ball in the sup-norm in $\R^d$, i.e. the cube
  \[ \Big\{ \vec x=(x_1,\ldots,x_d)\in\R^d \ \Big|\ |x_i|\leq1,\ i=1,\ldots,d\, \Big\} \]
  and set
  \begin{equation} \label{eq:Delta_d_definition}
    \Delta_d=\frac{1}{2^{d-1}\sqrt d}\vol_{d-1}\Big\{ \vec x\in\cB_\infty^d \,\Big|\, \sum_{i=1}^dx_i=0 \Big\},
  \end{equation}
  where $\vol_{d-1}(\cdot)$ denotes the $(d-1)$-dimensional Lebesgue measure.

  The factor mentioned above equals $\Delta_d^{-\frac{1}{d-1}}$. Due to Vaaler's and Ball's theorems
  (see \cite{vaaler}, \cite{ball}) the volume of each $(d-1)$-dimensional central section of
  $\cB_\infty^d$ is bounded between $2^{d-1}$ and $2^{d-1}\sqrt2$, so we have
  $\sqrt{d/2}\leq\Delta_d^{-1}\leq\sqrt d$. Hence $\Delta_d^{-\frac{1}{d-1}}$ is indeed less than
  $d-1$ and tends to $1$ as $d\to\infty$. We shall return to the properties of this quantity in
  Section \ref{sec:section_dual}, and now we are ready to formulate the main result of this paper.

  \begin{theorem} \label{t:multitrans}
    If there are $\vec x\in\Z^m\backslash\{\vec 0\}$, $\vec y\in\Z^n$, such that
    \begin{equation} \label{eq:multitrans_hypothesis}
      \Pi'(\vec x)\leq X,\qquad\Pi(\Theta\vec x-\vec y)\leq U,
    \end{equation}
    $0<U<1\leq X$, then there are $\vec x\in\Z^m$, $\vec y\in\Z^n\backslash\{\vec 0\}$, such that
    \begin{equation} \label{eq:multitrans_statement}
      \Pi'(\vec y)\leq Y,\qquad\Pi(\tr\Theta\vec y-\vec x)\leq V,
    \end{equation}
    \begin{equation} \label{eq:multitrans_supnorm_is_small}
      |\tr\Theta\vec y-\vec x|_\infty\leq\Delta_dV^mY^n,
    \end{equation}
    where
    \begin{equation} \label{eq:Y_V_in_terms_of_X_U}
      Y=\Delta_d^{-\frac{1}{d-1}}\big(X^mU^{1-m}\big)^{\frac{1}{d-1}},\quad
      V=\Delta_d^{-\frac{1}{d-1}}\big(X^{1-n}U^n\big)^{\frac{1}{d-1}},
    \end{equation}
    and $d=n+m$.
  \end{theorem}

  \section{Corollaries}

  \begin{corollary} \label{cor:bad_transference}
    $\Theta$ is multiplicatively badly approximable if and only if $\tr\Theta$ is multiplicatively
    badly approximable.
  \end{corollary}

  \begin{proof}
    It follows from Theorem \ref{t:multitrans} that if the inequality
    \[ \Pi'(\vec x)^m\Pi(\Theta\vec x-\vec y)^n\leq C \]
    with $C<1$ has a solution in $\vec x\in\Z^m\backslash\{\vec 0\}$, $\vec y\in\Z^n$, then there are
    $\vec x\in\Z^m$, $\vec y\in\Z^n\backslash\{\vec 0\}$, such that
    \[ \Pi'(\vec y)^n\Pi(\tr\Theta\vec y-\vec x)^m\leq\Delta_d^{-\frac{d}{d-1}}C^{\frac{1}{d-1}}. \]
    Hence, setting
    \[ \mu_1=\inf_{\begin{subarray}{c} (\vec x,\vec y)\in\Z^m\oplus\Z^n \\ \vec x\neq\vec 0 \end{subarray}}
       \Pi'(\vec x)^m\Pi(\Theta\vec x-\vec y)^n \]
    and
    \[ \mu_2=\inf_{\begin{subarray}{c} (\vec x,\vec y)\in\Z^m\oplus\Z^n \\ \vec y\neq\vec 0 \end{subarray}}
       \Pi'(\vec y)^n\Pi(\tr\Theta\vec y-\vec x)^m, \]
    we get the inequalities
    \[ \Delta_d^d\mu_2^{d-1}\leq\mu_1\leq\Delta_d^{-\frac{d}{d-1}}\mu_2^{\frac{1}{d-1}}. \]
    Particularly, $\mu_1>0$ if and only if $\mu_2>0$.
  \end{proof}

  \begin{corollary} \label{cor:dysonlike}
    \[ \mbeta(\tr\Theta)\geq\frac{n\mbeta(\Theta)+n-1}{(m-1)\mbeta(\Theta)+m}\,. \]
  \end{corollary}

  \begin{proof}
    Let $\gamma$ be a positive real number, $\gamma<\mbeta(\Theta)$. By the hypothesis there are infinitely many pairs $\vec x\in\Z^m\backslash\{\vec 0\}$, $\vec y\in\Z^n$ satisfying
    \begin{equation} \label{eq:mbeta_0}
      \Pi(\Theta\vec x-\vec y)\leq\Pi'(\vec x)^{-\gamma}.
    \end{equation}
    If infinitely many of these pairs have same $\vec x$-component, then for such an $\vec x$ the vector $\Theta\vec x$ should have an integer component. But then all the integer multiples of such pairs will satisfy \eqref{eq:mbeta}. Therefore, we may consider a sequence of pairs $\vec x_i\in\Z^m\backslash\{\vec 0\}$, $\vec y_i\in\Z^n$, $i\in\Z_+$, such that
    \[ \Pi'(\vec x_i)=t_i>1,\quad\Pi(\Theta\vec x_i-\vec y_i)\leq t_i^{-\gamma},\quad
       t_i\to\infty\text{\ \ as\ \ }i\to\infty. \]
    Applying Theorem \ref{t:multitrans} we get a sequence of pairs $\vec x_i'\in\Z^m$, $\vec y_i'\in\Z^n\backslash\{\vec 0\}$, such that
    \begin{equation} \label{eq:primes_inequalities}
      \Pi'(\vec y_i')\leq\Delta_d^{-\frac{1}{d-1}}t_i^{\frac{(m-1)\gamma+m}{d-1}},\quad
      \Pi(\tr\Theta\vec y_i'-\vec x_i')\leq\Delta_d^{-\frac{1}{d-1}}t_i^{-\frac{n\gamma+(n-1)}{d-1}}.
    \end{equation}
    Hence for each $i$ we have
    \begin{equation} \label{eq:mbeta_i_inequality}
      \Pi(\tr\Theta\vec y_i'-\vec x_i')\leq\Pi'(\vec y_i')^{-\gamma_i},
    \end{equation}
    where
    \begin{equation*} 
      \gamma_i=\frac{n\gamma+(n-1)+\varkappa_i}{(m-1)\gamma+m-\varkappa_i}\,,\quad\varkappa_i=\frac{\ln\Delta_d}{\ln t_i}\,.
    \end{equation*}
    If the pairs $(\vec x_i',\vec y_i')$ coincide for infinitely many values of $i$, then for such repeating pairs the vector $\tr\Theta\vec y_i'-\vec x_i'$ should have a zero component, since the righthand side of the second inequality in \eqref{eq:primes_inequalities} tends to zero as $i$ tends to infinity. But then \eqref{eq:mbeta_i_inequality} holds with $(\vec x_i',\vec y_i')$ substituted by any integer multiple of $(\vec x_i',\vec y_i')$. Thus, we may suppose that there are infinitely many distinct pairs among $(\vec x_i',\vec y_i')$. Then we immediately get
    \[ \mbeta(\tr\Theta)\geq\sup_{\gamma<\mbeta(\Theta)}\limsup_{i\in\Z_+}\gamma_i=\frac{n\mbeta(\Theta)+(n-1)}{(m-1)\mbeta(\Theta)+m}\,. \]
  \end{proof}

  Notice that in the proof of Corollary \ref{cor:dysonlike} we didn't use
  \eqref{eq:multitrans_supnorm_is_small} at all. But in the case when $n=1$ the inequality $\Pi'(\vec
  y)\leq Y$ for a non-zero integer $\vec y$ means just that $|\vec y|\leq Y$, which, combined with
  \eqref{eq:multitrans_supnorm_is_small} gives some information concerning ordinary Diophantine
  approximation for $\tr\Theta$. Namely, we get

  \begin{corollary} \label{cor:tr_beta_and_mbeta}
    If $n=1$, then
    \begin{equation} \label{eq:tr_beta_and_mbeta}
      \beta(\tr\Theta)\geq\frac{\mbeta(\Theta)-m}{(m-1)\mbeta(\Theta)+m}\,.
    \end{equation}
  \end{corollary}

  \begin{proof}
    Let $\gamma$ be as in the proof of Corollary \ref{cor:bad_transference}. Same as in that proof, we can get a sequence of pairs $\vec x_i'\in\Z^m$, $\vec y_i'\in\Z\backslash\{\vec 0\}$, such that
    \begin{equation} 
      |\vec y_i'|\leq\Delta_d^{-\frac{1}{m}}t_i^{\frac{(m-1)\gamma+m}{m}},\quad
      |\tr\Theta\vec y_i'-\vec x_i'|_\infty\leq\Delta_d^{-\frac{1}{m}}t_i^{\frac{m-\gamma}{m}},
    \end{equation}
    with $t_i\to\infty$ as $i\to\infty$. Thus, for each $i$ we have
    \begin{equation} 
      |\tr\Theta\vec y_i'-\vec x_i'|_\infty\leq|\vec y_i'|^{-\gamma_i},
    \end{equation}
    where now
    \begin{equation*} 
      \gamma_i=\frac{\gamma-m+\varkappa_i}{(m-1)\gamma+m-\varkappa_i}\,,\quad\varkappa_i=\frac{\ln\Delta_d}{\ln t_i}\,.
    \end{equation*}
    Hence
    \[ \beta(\tr\Theta)\geq\sup_{\gamma<\mbeta(\Theta)}\limsup_{i\in\Z_+}\gamma_i=\frac{\mbeta(\Theta)-m}{(m-1)\mbeta(\Theta)+m}\,. \]
  \end{proof}

  The inequality \eqref{eq:tr_beta_and_mbeta} is stronger than the trivial bound
  $\beta(\tr\Theta)\geq1/m$ for $\mbeta(\Theta)>m+m^2$ and is stronger than Khintchine's inequality
  \[ \beta(\tr\Theta)\geq\frac{\beta(\Theta)}{(m-1)\beta(\Theta)+m} \]
  whenever $\mbeta(\Theta)>m\beta(\Theta)+m$.

  Combining Corollaries \ref{cor:dysonlike} and \ref{cor:tr_beta_and_mbeta} we get

  \begin{corollary} \label{cor:beta_and_mbeta}
    If $m=1$, then
    \begin{equation} \label{eq:beta_and_mbeta}
      \mbeta(\Theta)\geq\beta(\Theta)\geq\frac{n\mbeta(\Theta)-1}{n(n-1)\mbeta(\Theta)+n^2-n+1}\,.
    \end{equation}
  \end{corollary}

  The latter inequality is stronger than the trivial bound $\beta(\Theta)\geq1/n$ for
  $\mbeta(\Theta)>n+1/n$.

  We have no reason to believe that \eqref{eq:tr_beta_and_mbeta} or \eqref{eq:beta_and_mbeta} cannot
  be improved. For instance, the latter is obtained by ``double-transfer'', i.e. by transition to the
  dual space, and backwards. It is natural to expect that such a method should lose something. To
  illustrate this we give another corollary in the case $m=1$, $n=2$, i.e. in the case of the
  Littlewood conjecture. We denote by $\|\cdot\|$ the distance to the nearest integer.

%
%
%

  \begin{corollary} \label{cor:littlewood_and_supnorm}
    If for real numbers $\alpha$ and $\beta$ there are infinitely many $q\in\Z_+$, such that
    \[ q\,\|q\alpha\|\,\|q\beta\|\leq\mu, \]
    then there are infinitely many $q\in\Z_+$, such that
    \begin{equation} \label{eq:littlewood_and_supnorm_prod}
      q\,\|q\alpha\|\,\|q\beta\|\leq(4/3)^{9/4}\mu^{1/4},
    \end{equation}
    \begin{equation} \label{eq:littlewood_and_supnorm_supnorm}
      \max(\|q\alpha\|,\|q\beta\|)\leq(4/3)^{5/4}\mu^{1/4}.
    \end{equation}
  \end{corollary}

  As Prof.\,Moshchevitin noticed, such a statement can be easily proved directly with the help of the Dirichlet theorem, even with both constants in \eqref{eq:littlewood_and_supnorm_prod} and \eqref{eq:littlewood_and_supnorm_supnorm} substituted by $1$, which only strengthens the statement.
  However, it is not clear whether the exponent $1/4$ can be improved in any of the inequalities \eqref{eq:littlewood_and_supnorm_prod} and \eqref{eq:littlewood_and_supnorm_supnorm}, or even whether $\mu^{1/4}$ can be substituted by $o(\mu^{1/4})$. So, in case this exponent is best possible, it is quite curious that it is given even by the method of ``double transfer''.

  \section{Arbitrary functions}

  Considering only exponents when investigating the asymptotic behaviour of some quantity does not
  allow to detect any intermediate growth. It appears, however, that Theorem \ref{t:multitrans}
  allows to work not only with the (multiplicative) Diophantine exponents, but also with arbitrary
  functions satisfying some natural growth conditions. In this Section we formulate the corresponding
  statement and give examples of how to ``transfer'' the information about intermediate growth.

  Let $\psi:\R_+\to\R_+$ be an arbitrary function. By analogy with Definitions \ref{def:beta}
  \ref{def:mbeta}, we give the following

  \begin{definition} \label{def:beta_f}
    We call $\Theta$ \emph{$\psi$-approximable}, if there are infinitely many $\vec x\in\Z^m$, $\vec
    y\in\Z^n$ satisfying the inequality
    \[ |\Theta\vec x-\vec y|_\infty\leq\psi(|\vec x|_\infty). \]
  \end{definition}

  \begin{definition} \label{def:mbeta_f}
    We call $\Theta$ \emph{multiplicatively $\psi$-approximable}, if there are infinitely many $\vec
    x\in\Z^m\backslash\{\vec 0\}$, $\vec y\in\Z^n$ satisfying the inequality
    \begin{equation} \label{eq:mbeta_f}
      \Pi(\Theta\vec x-\vec y)\leq\psi(\Pi'(\vec x)).
    \end{equation}
  \end{definition}

  Obviously, $\beta(\Theta)$ equals the supremum of real numbers $\gamma$, such that $\Theta$ is
  $t^{-\gamma}$-approximable. Similarly, $\mbeta(\Theta)$ equals the supremum of real numbers
  $\gamma$, such that $\Theta$ is multiplicatively $t^{-\gamma}$-approximable.

%

  \begin{theorem} \label{t:arbitrary_functions}
    Let $\psi:\R_+\to\R_+$ be an arbitrary function satisfying the conditions
    \[ \psi(t)<1\quad\text{ for all $t$ large enough}, \]
    and let
    \[ t^{\frac{1-n}{n}}\psi(t)\to0\quad\text{ as }\quad t\to\infty. \]
    Suppose that the function
    \[ f(t)=\Delta_d^{-\frac{1}{d-1}}\big(t^m\psi(t)^{1-m}\big)^{\frac{1}{d-1}} \]
    is invertible. Let
    \[ g(t)=\Delta_d^{-\frac{1}{d-1}}\big(t^{1-n}\psi(t)^n\big)^{\frac{1}{d-1}}, \]
    and let
    \[ \phi(t)=g(f^-(t)), \]
    where $f^-$ is the inverse of $f$.

    Let $\Theta$ be multiplicatively $\psi$-approximable. Then $\tr\Theta$ is multiplicatively
    $\phi$-approximable.
  \end{theorem}

  \begin{proof}
    By the hypothesis there are infinitely many pairs $\vec x\in\Z^m\backslash\{\vec 0\}$, $\vec
    y\in\Z^n$ satisfying \eqref{eq:mbeta_f}. If infinitely many of these pairs have same $\vec
    x$-component, then for such an $\vec x$ the vector $\Theta\vec x$ should have an integer component.
    But then all the integer multiples of such pairs will satisfy \eqref{eq:mbeta_f}. Hence we may
    consider a sequence of pairs $\vec x_i\in\Z^m\backslash\{\vec 0\}$, $\vec y_i\in\Z^n$, $i\in\Z_+$,
    such that
    \[ \Pi'(\vec x_i)=t_i,\quad\Pi(\Theta\vec x_i-\vec y_i)\leq\psi(t_i),\quad
       t_i\to\infty\text{\ \ as\ \ }i\to\infty. \]
    Then, starting with some $i_0$ all the $\psi(t_i)$ are less than $1$, so, applying Theorem
    \ref{t:multitrans} for each pair $\vec x_i$, $\vec y_i$, $i\geq i_0$, we get a pair $\vec
    x_i'\in\Z^m$, $\vec y_i'\in\Z^n\backslash\{\vec 0\}$, such that
    \[ \Pi'(\vec y_i')\leq f(t_i),\quad\Pi(\tr\Theta\vec y_i'-\vec x_i')\leq g(t_i),\quad
       g(t_i)\to0\text{\ \ as\ \ }i\to\infty. \]
    Hence for each $i\geq i_0$ we have
    \begin{equation} \label{eq:prod_leq_phi}
      \Pi(\tr\Theta\vec y_i'-\vec x_i')\leq\phi(\Pi'(\vec y_i')).
    \end{equation}
    If the pairs $(\vec x_i',\vec y_i')$ coincide for infinitely many values of $i$, then for such repeating pairs the vector $\tr\Theta\vec y_i'-\vec x_i'$ should have a zero component, since $g(t_i)\to0$ as $i\to\infty$. But then \eqref{eq:prod_leq_phi} holds with $(\vec x_i',\vec y_i')$ substituted by any integer multiple of $(\vec x_i',\vec y_i')$. Thus, we may suppose that there are infinitely many distinct pairs among $(\vec x_i',\vec y_i')$, which means that $\tr\Theta$ is multiplicatively $\phi$-approximable.
  \end{proof}

  \begin{remark}
    In case $n=1$ the proof of Theorem \ref{t:arbitrary_functions} can be easily modified to show that
    $\tr\Theta$ is $\chi$-approximable (in the ordinary sense), where
    \[ \chi(t)=h(f^-(t)),\quad h(t)=\Delta_dt^m\psi(t)^n. \]
  \end{remark}

  \begin{remark}
    It is also easy to see that Corollaries \ref{cor:dysonlike}, \ref{cor:tr_beta_and_mbeta} can be
    derived from Theorem \ref{t:arbitrary_functions} with $\psi$ and $\phi$ set to be the corresponding
    exponential functions.
  \end{remark}

  Let us give an example of a transference statement sensible to logarithmic growth. We give it in
  the case of the Littlewood conjecture. Set $c=\Delta_3^{-1/2}=\sqrt3/2$.

%

  \begin{corollary}
    Given $\alpha,\beta\in\R$ suppose that there are infinitely many $q\in\Z_+$, such that
    \[ \|q\alpha\|\,\|q\beta\|\leq\frac{1}{q\log q}\,. \]
    Then there are infinitely many $(p,q)\in\Z^2\backslash\{\vec 0\}$, such that
    \[ \|p\alpha+q\beta\|\leq\frac{c^3}{t^2\sqrt{2\log(t/c)}}\,,\quad
       t=\sqrt{\max(1,|p|)\cdot\max(1,|q|)}. \]
  \end{corollary}

  Here is its analogue in the opposite direction. As before, $c=\Delta_3^{-1/2}=\sqrt3/2$. Denote by $\rho(t)$ the function, inverse to $ct^2\sqrt{\log(1+t)}$.

%

  \begin{corollary}
    Given $\alpha,\beta\in\R$ suppose that there are infinitely many $(p,q)\in\Z^2\backslash\{\vec
    0\}$, such that
    \[ \|p\alpha+q\beta\|\leq\frac{1}{t^2\log(1+t)}\,,\quad
       t=\sqrt{\max(1,|p|)\cdot\max(1,|q|)}. \]
    Then there are infinitely many $q\in\Z_+$, such that
    \[ \|q\alpha\|\,\|q\beta\|\leq\frac{c^2}{\rho(q)^2\log(1+\rho(q))}\,. \]
  \end{corollary}

  \section{Section-dual set} \label{sec:section_dual}

   Let $\cS^{d-1}$ denote the Euclidean unit sphere in $\R^d$. Let $\vol_k(\cdot)$ denote the
   $k$-dimensional Lebesgue measure, and let $\langle\,\cdot\,,\cdot\,\rangle$ denote the inner
   product. For each measurable set $M\subset\R^d$ and each $\vec e\in \cS^{d-1}$ we set
  \[ \vol_{\vec e}(M)=
     \vol_{d-1}\Big\{ \vec x\in M \,\Big|\, \langle\vec e,\vec x\rangle=0 \Big\}. \]

  \begin{definition}
    Let $M$ be a measurable subset of $\R^d$. We call the set
    \[ M^\wedge=\{\, \lambda\vec e\ |\ \vec e\in\cS^{d-1},\ 0\leq\lambda\leq2^{1-d}\vol_{\vec e}(M)\, \} \]
    \emph{section-dual} for $M$.
  \end{definition}

  This construction is the main tool we use to prove Theorem \ref{t:multitrans}. In \cite{german} the
  following properties of section-dual sets are proved:

  \begin{lemma} \label{l:wedge_properties}
    ${}$ \\
    \textup{(i)} Let $M$ be convex and $\vec 0$-symmetric. Let $M^\wedge\cap\Z^d\backslash\{\vec 0\}\neq\varnothing$.
    Then $M\cap\Z^d\backslash\{\vec 0\}\neq\varnothing$. \\
    \textup{(ii)} If $M$ is convex, then so is $M^\wedge$. \\
    \textup{(iii)} Let $A$ be a non-degenerate $d\times d$ real matrix. Then $(AM)^\wedge=A'(M^\wedge)$,
    where $A'$ denotes the cofactor matrix of $A$. \\
    \textup{(iv)} $(\cB_\infty^d)^\wedge$ contains the cube $\Delta_d\cB_\infty^d$.
  \end{lemma}

  The constant $\Delta_d$ is defined by \eqref{eq:Delta_d_definition} and is the maximal number, such
  that the statement \textup{(iv)} of Lemma \ref{l:wedge_properties} holds.

  \section{Monotonicity of $\Delta_d$}

  Let us show that $\Delta_d$ decreases as $d$ increases, for we shall need this fact to prove Theorem \ref{t:multitrans}.

  \begin{lemma} \label{l:vol_and_sec}
    Let $M\subset\R^k$ be a convex $k$-dimensional $\vec 0$-symmetric body. Let $h$ be the thickness of
    $M$ with respect to $\vec e\in\cS^{k-1}$, i.e. the supremum of the quantity $2\langle\vec e,\vec
    x\rangle$ over the $\vec x\in M$. Then
    \[ \vol_kM\leq h\vol_{\vec e}(M). \]
  \end{lemma}

  \begin{proof}
    Set
    \[ \phi(t)=\vol_{k-1}\Big\{ \vec x\in M \,\Big|\, \langle\vec e,\vec x\rangle=t \Big\} \]
    Due to the Brunn-Minkowski inequality $\phi(t)^\frac{1}{k-1}$ is $t$-concave (see
    \cite{bonnesen_fenchel}). But $\phi(t)=\phi(-t)$, so,
    \[ \phi(0)\geq\left(\frac{\phi(t)^\frac{1}{k-1}+\phi(-t)}{2}^\frac{1}{k-1}\right)^{k-1}=
       \phi(t). \]
    Therefore,
    \[ \vol_kM=\int_{t=-h/2}^{t=h/2}\phi(t)\leq h\phi(0)=h\vol_{\vec e}(M). \]
  \end{proof}

  \begin{lemma} \label{l:Delta_d_decreases}
    $\Delta_d\leq\Delta_{d-1}$.
  \end{lemma}


  \begin{proof}
    Let us consider $\cB_\infty^{d-1}$ as the subset of $\cB_\infty^d$ consisting of points with zero
    $d$-th coordinate. Set
    \[ M_k=\Big\{ \vec x\in\cB_\infty^k \,\Big|\, \sum_{i=1}^kx_i=0 \Big\},\quad k=d,d-1. \]
    Set also
    \[ \vec e=\frac{1}{\sqrt{d(d-1)}}\tr{\big(1,\ldots,1,1-d\big)}\in\cS^{d-1}. \]
    For every $\vec x\in M_d$ we have
    \[ \langle\vec e,\vec x\rangle=\frac{-dx_d}{\sqrt{d(d-1)}}\leq\sqrt{\frac{d}{d-1}}\,. \]
    Hence thickness of $M_d$ with respect to $\vec e$ does not exceed $2\sqrt{d/(d-1)}$. Applying Lemma
    \ref{l:vol_and_sec}, we get
    \[ \Delta_d=\frac{1}{2^{d-1}\sqrt d}\vol_{d-1}(M_d)\leq
       \frac{1}{2^{d-2}\sqrt{d-1}}\vol_{d-2}(M_{d-1})=\Delta_{d-1}, \]
    which proves the Lemma.
  \end{proof}

  \section{$d$-dimensional setting}

  As before, we use $d$ as $m+n$. Let us denote by $\pmb\ell_1,\ldots,\pmb\ell_m,\vec
  e_{m+1},\ldots,\vec e_d$ the columns of the matrix
  \[ T=
     \begin{pmatrix}
       E_m & 0 \\
       -\Theta & E_n
     \end{pmatrix}, \]
  where $E_m$ and $E_n$ are the corresponding unity matrices, and by $\vec e_1,\ldots,\vec
  e_m,\pmb\ell_{m+1},\ldots,\pmb\ell_d$ the columns of
  \[ T'=
     \begin{pmatrix}
       E_m & \tr\Theta \\
       0 & E_n
     \end{pmatrix}. \]
  We obviously have $T\tr{(T')}=E_d$, so the bases $\pmb\ell_1,\ldots,\pmb\ell_m,\vec
  e_{m+1},\ldots,\vec e_d$ and $\vec e_1,\ldots,\vec e_m,\pmb\ell_{m+1},\ldots,\pmb\ell_d$ are dual.
  Therefore, the subspaces
  \[ \cL^m=\spanned_\R(\pmb\ell_1,\ldots,\pmb\ell_m),\qquad
     \cL^n=\spanned_\R(\pmb\ell_{m+1},\ldots,\pmb\ell_d) \]
  are orthogonal. More than that, $\cL^m=(\cL^n)^\perp$ and
  \[ \cL^m=\{ \vec z\in\R^d \,|\, \langle\pmb\ell_{m+i},\vec z\rangle=0,\ i=1,\ldots,n \},\qquad
     \cL^n=\{ \vec z\in\R^d \,|\, \langle\pmb\ell_j,\vec z\rangle=0,\ j=1,\ldots,m \}, \]
  where, as before, $\langle\,\cdot\,,\cdot\,\rangle$ denotes the inner product.
  Thus, a point $\vec z=(\vec x,-\vec y)\in\R^m\oplus\R^n$ lies in $\cL^m$ if and only if $\Theta\vec
  x=\vec y$, and a point $\vec z=(\vec x,\vec y)\in\R^m\oplus\R^n$ lies in $\cL^n$ if and only if
  $\tr\Theta\vec y=\vec x$. That is the spaces $\cL^m$ and $\cL^n$ are isomorphic to the spaces of
  solutions of the systems \eqref{eq:the_system} and \eqref{eq:the_transposed_system}, respectively.
  Thus, given an integer point close to $\cL^m$ we are to find an integer point close to $\cL^n$,
  understanding closeness in terms of the geometric mean.

  \section{Proof of Theorem \ref{t:multitrans}}

  Set
  \begin{equation} \label{eq:H_def}
    \cH=\Big\{ \vec z=(\vec x,-\vec y)\in\R^m\oplus\R^n \ \Big|\ \vec x,\vec y
    \text{ satisfy }
    \eqref{eq:multitrans_hypothesis} \Big\},
  \end{equation}
  \[ \cH'=\Big\{ \vec z=(\vec x,-\vec y)\in\cH \ \Big|\ |\Theta\vec x-\vec y|_\infty\leq\Delta_dV^mY^n \Big\},\, \]
  and
  \begin{equation} \label{eq:H_hat_def}
    \widehat\cH=\Big\{ \vec z=(\vec x,\vec y)\in\R^m\oplus\R^n \ \Big|\ \vec x,\vec y
    \text{ satisfy }
    \eqref{eq:multitrans_statement},\,\eqref{eq:multitrans_supnorm_is_small} \Big\}.
  \end{equation}
  We must prove that if $\cH$ contains a non-zero integer point, then so does $\widehat\cH$.

  $\cH$ is the union of parallelepipeds
  \begin{align*}
    M_{\pmb\lambda,\pmb\mu}=\Big\{ \vec z\in\R^d \,\Big|\ &
    |\langle\vec e_j,\vec z\rangle|\leq\lambda_j,\ \ j=1,\ldots,m, \\ &
    |\langle\pmb\ell_{m+i},\vec z\rangle|\leq\mu_i,\ \ i=1,\ldots,n \Big\}
  \end{align*}
  over all the tuples
  $(\pmb\lambda,\pmb\mu)=(\lambda_1,\ldots,\lambda_m,\mu_1,\ldots,\mu_n)\in\R_+^d$ satisfying the
  conditions
  \begin{equation} \label{eq:lambda_mu_X_U}
    \lambda_1\ldots\lambda_m=X^m,\quad\mu_1\ldots\mu_n=U^n,
  \end{equation}
  \begin{equation} \label{eq:lambda_X_U_bounds}
    \min_{1\leq j\leq m}\lambda_j\geq1.
  \end{equation}
  $\cH'$ is obtained if we supplement \eqref{eq:lambda_mu_X_U} with
  \begin{equation} \label{eq:mu_X_U_bounds}
    \max_{1\leq i\leq n}\mu_i\leq\Delta_dV^mY^n.
  \end{equation}
  Similarly, $\widehat\cH$ is the union of the parallelepipeds
  \begin{align*}
    \widehat M_{\pmb\lambda,\pmb\mu}=\Big\{ \vec z\in\R^d \,\Big|\ &
    |\langle\pmb\ell_j,\vec z\rangle|\leq \lambda_j,\ \ j=1,\ldots,m, \\ &
    |\langle\vec e_{m+i},\vec z\rangle|\leq \mu_i,\ \ i=1,\ldots,n \Big\}.
  \end{align*}
  over all the tuples
  $(\pmb\lambda,\pmb\mu)=(\lambda_1,\ldots,\lambda_m,\mu_1,\ldots,\mu_n)\in\R_+^d$ satisfying the
  conditions
  \begin{equation} \label{eq:lambda_mu_V_Y}
    \lambda_1\ldots\lambda_m=V^m,\quad\mu_1\ldots\mu_n=Y^n,
  \end{equation}
  \begin{equation} \label{eq:lambda_mu_V_Y_bounds}
    \max_{1\leq j\leq m}\lambda_j\leq\Delta_dV^mY^n,\quad
    \min_{1\leq i\leq n}\mu_i\geq1.
  \end{equation}

  For each tuple $(\pmb\lambda,\pmb\mu)$ let us denote by $D_{\pmb\lambda,\pmb\mu}$ the $d\times d$
  diagonal matrix with diagonal elements $\lambda_1^{-1},\ldots,\lambda_m^{-1}$,
  $\mu_1^{-1},\ldots,\mu_n^{-1}$. Set
  \[ A_{\pmb\lambda,\pmb\mu}=TD_{\pmb\lambda,\pmb\mu} \]
  For the cofactor matrix $A'_{\pmb\lambda,\pmb\mu}$ we have
  \[ A'_{\pmb\lambda,\pmb\mu}=T'D_{\pmb\lambda',\pmb\mu'}, \]
  where $(\pmb\lambda',\pmb\mu')=(\lambda_1',\ldots,\lambda_m',\mu_1',\ldots,\mu_n')$,
  \[ \lambda_j'=\lambda_j^{-1}\prod_{k=1}^m\lambda_k\prod_{l=1}^n\mu_l,\qquad
     \mu_i'=\mu_i^{-1}\prod_{k=1}^m\lambda_k\prod_{l=1}^n\mu_l. \]
  Notice that the correspondence $(\pmb\lambda,\pmb\mu)\mapsto(\pmb\lambda',\pmb\mu')$ is a
  one-to-one map from $\R_+^d$ onto itself. Besides that, if $(\pmb\lambda,\pmb\mu)$ satisfies
  \eqref{eq:lambda_mu_V_Y},\,\eqref{eq:lambda_mu_V_Y_bounds}, then due to
  \eqref{eq:Y_V_in_terms_of_X_U}
  \[ \lambda_1'\ldots\lambda_m'=V^{m(m-1)}Y^{mn}=\Delta_d^{-m}X^m \]
  and
  \[ \mu_1'\ldots\mu_n'=V^{mn}Y^{n(n-1)}=\Delta_d^{-n}U^n, \]
  that is $(\Delta_d\pmb\lambda',\Delta_d\pmb\mu')$ satisfies
  \eqref{eq:lambda_mu_X_U},\,\eqref{eq:lambda_X_U_bounds},\,\eqref{eq:mu_X_U_bounds}. Hence we
  conclude that the correspondence
  $(\pmb\lambda,\pmb\mu)\mapsto(\Delta_d\pmb\lambda',\Delta_d\pmb\mu')$ is a bijective map from the
  surface in $\R_+^d$ defined by \eqref{eq:lambda_mu_V_Y},\,\eqref{eq:lambda_mu_V_Y_bounds} onto the
  surface in $\R_+^d$ defined by
  \eqref{eq:lambda_mu_X_U},\,\eqref{eq:lambda_X_U_bounds},\,\eqref{eq:mu_X_U_bounds}. This
  correspondence generates naturally a bijection between the sets of parallelepipeds
  $M_{\pmb\lambda,\pmb\mu}$ and $\widehat M_{\pmb\lambda,\pmb\mu}$ which cover $\cH'$ and
  $\widehat\cH$, respectively. On the other hand, if $(\pmb\lambda,\pmb\mu)$ satisfies
  \eqref{eq:lambda_mu_V_Y}, we have
  \[ \widehat M_{\pmb\lambda,\pmb\mu}=(A_{\pmb\lambda,\pmb\mu}^\ast)^{-1}\cB_\infty^d\quad\text{ and }\quad
     M_{\Delta_d\pmb\lambda',\Delta_d\pmb\mu'}=\Delta_d((A'_{\pmb\lambda,\pmb\mu})^\ast)^{-1}\cB_\infty^d, \]
  which, in view of the statements \textup{(iii)}, \textup{(iv)} of Lemma \ref{l:wedge_properties},
  implies the inclusion
  \[ M_{\Delta_d\pmb\lambda',\Delta_d\pmb\mu'}\subset(\widehat M_{\pmb\lambda,\pmb\mu})^\wedge. \]
  Hence, if $\cH'$ contains a non-zero integer point, then there exists a tuple
  $(\pmb\lambda,\pmb\mu)$ satisfying \eqref{eq:lambda_mu_V_Y},\,\eqref{eq:lambda_mu_V_Y_bounds}, such
  that $(\widehat M_{\pmb\lambda,\pmb\mu})^\wedge$ contains a non-zero integer point. But then so
  does $\widehat M_{\pmb\lambda,\pmb\mu}$, as follows from the statement \textup{(i)} of Lemma
  \ref{l:wedge_properties}. If the $\vec y$-component of this point happened to be zero, then we
  should have $V\geq1$. In this case for any $\vec y\in\Z^n$ there is an $\vec x\in\Z^m$, such that
  the second inequality of \eqref{eq:multitrans_statement} holds. And since $Y>1$, the first
  inequality of \eqref{eq:multitrans_statement} has a non-zero integer solution. So, in each case
  there is an integer point $\vec z=(\vec x,\vec y)$ with a non-zero $\vec y$-component in $\widehat
  M_{\pmb\lambda,\pmb\mu}$, i.e. in $\widehat\cH$.

  We have proved the Theorem in the case when $\cH'$ contains a non-zero integer point. But by the
  hypothesis only $\cH$ is guaranteed to contain such a point, not necessarily $\cH'$. However, in
  case $n=1$ we have $\cH=\cH'$, which follows from \eqref{eq:Y_V_in_terms_of_X_U} and the
  inequalities $U<1\leq X$, $\Delta_d<1$. Hence the Theorem is proved for $n=1$. Now, having the
  latter as the base of induction, let us establish the induction step, i.e. let us suppose that the
  Theorem holds for $m$ and $n-1$ and prove it for $m$ and $n$.

  If $\cH'\cap\Z^d=\{\vec 0\}$, then there is a non-zero integer point $\vec z=(\vec x,-\vec y)$ in
  $\cH\backslash\cH'$. Without loss of generality we may assume that
  \[ |\langle\pmb\ell_d,\vec z\rangle|>\Delta_dV^mY^n. \]
  Then
  \[ \prod_{i=1}^{n-1}|\langle\pmb\ell_{m+i},\vec z\rangle|<(\Delta_dV^mY^n)^{-1}
     \prod_{i=1}^n|\langle\pmb\ell_{m+i},\vec z\rangle|\leq(\Delta_dV^mY^n)^{-1}U^n=
     \Delta^{\frac{1}{d-1}}X^{\frac{-m}{d-1}}U^{\frac{n(d-2)}{d-1}}. \]
  Set
  \begin{equation} \label{eq:U_1}
    U_1=\Big(\Delta^{\frac{1}{d-1}}X^{\frac{-m}{d-1}}U^{\frac{n(d-2)}{d-1}}\Big)^{\frac{1}{n-1}}.
  \end{equation}
  By the hypothesis of the induction there are $\vec x_1\in\Z^m$, $\vec
  y_1\in\Z^{n-1}\backslash\{\vec 0\}$, such that
  \begin{equation*} 
    \Pi'(\vec y_1)\leq Y_1,\qquad\Pi(\tr\Theta_1\vec y_1-\vec x_1)\leq V_1,
  \end{equation*}
  \begin{equation*} 
    |\tr\Theta_1\vec y_1-\vec x_1|_\infty\leq\Delta_{d-1}V_1^mY_1^{n-1},
  \end{equation*}
  where
  \begin{equation*} 
    Y_1=\Delta_{d-1}^{-\frac{1}{d-2}}\big(X^mU_1^{1-m}\big)^{\frac{1}{d-2}},\quad
    V_1=\Delta_{d-1}^{-\frac{1}{d-2}}\big(X^{2-n}U_1^{n-1}\big)^{\frac{1}{d-2}}
  \end{equation*}
  and
  \[ \Theta_1=
     \begin{pmatrix}
       \theta_{1,1} & \cdots & \theta_{1,m} \\
         \vdots & \ddots & \vdots \\
       \theta_{n-1,1} & \cdots & \theta_{n-1,m}
     \end{pmatrix}. \]
  It follows from \eqref{eq:U_1} that
  \[ Y_1^{n-1}=\big(\Delta_{d-1}^{-1}\Delta_d\big)^{\frac{n-1}{d-2}}Y^n,\quad
     V_1^m=\big(\Delta_{d-1}^{-1}\Delta_d\big)^{\frac{m}{d-2}}V^m. \]
  Hence, taking into account Lemma \ref{l:Delta_d_decreases}, we see that
  \[ Y_1^{n-1}\leq Y^n,\quad V_1^m\leq V^m,\quad\Delta_{d-1}V_1^mY_1^{n-1}\leq\Delta_dV^mY^n, \]
  so, supplementing $\vec y_1$ with a zero coordinate, we get a point $\vec y_2=(\vec y_1,0)\in\Z^n$,
  such that
  \begin{equation*} \label{eq:induction_statement_1}
    \Pi'(\vec y_2)\leq Y,\qquad\Pi(\tr\Theta\vec y_2-\vec x_1)\leq V,
  \end{equation*}
  \begin{equation*} \label{eq:induction_statement_2}
    |\tr\Theta\vec y_2-\vec x_1|_\infty\leq\Delta_dV^mY^n.
  \end{equation*}
  This provides the induction step and completes the proof of the Theorem.

  \section{Concerning uniform exponents}

  Same as in the case of ordinary Diophantine exponents, it is natural to consider the uniform analogue of $\mbeta(\Theta)$. We remind that in the ordinary case we have the following

  \begin{definition} \label{def:alpha}
    The supremum of real numbers $\gamma$, such that for each $t$ large enough there are $\vec
    x\in\Z^m\backslash\{\vec 0\}$, $\vec y\in\Z^n$ satisfying the inequalities
    \begin{equation*} 
      |\vec x|_\infty\leq t,\qquad|\Theta\vec x-\vec y|_\infty\leq t^{-\gamma},
    \end{equation*}
    is called the \emph{uniform (ordinary) Diophantine exponent} of $\Theta$ and is denoted by $\alpha(\Theta)$.
  \end{definition}

  In the multiplicative case we have

  \begin{definition} \label{def:malpha}
    The supremum of real numbers $\gamma$, such that for each $t$ large enough there are
    $\vec x\in\Z^m\backslash\{\vec 0\}$, $\vec y\in\Z^n$ satisfying the inequalities
    \begin{equation*} 
      \Pi'(\vec x)\leq t,\qquad\Pi(\Theta\vec x-\vec y)\leq t^{-\gamma},
    \end{equation*}
    is called the \emph{uniform multiplicative Diophantine exponent} of $\Theta$ and is denoted by $\malpha(\Theta)$.
  \end{definition}

  Obviously, $\mbeta(\Theta)\geq\malpha(\Theta)$, same as $\beta(\Theta)\geq\alpha(\Theta)$. Besides that, by analogy with \eqref{eq:ord_and_mult_trivial} and \eqref{eq:by_minkowski}, we have trivial inequalities

  \begin{equation*} 
    m/n\leq\alpha(\Theta)\leq\malpha(\Theta)\leq
    \begin{cases}
      m\alpha(\Theta),\quad\text{ if }n=1, \\
      +\infty,\qquad\ \text{ otherwise}.
    \end{cases}
  \end{equation*}

  The very same way we derived Corollary \ref{cor:dysonlike} from Theorem \ref{t:multitrans}, we get

  \begin{corollary} \label{cor:dysonlike_malpha}
    \begin{equation} \label{eq:dysonlike_malpha}
      \malpha(\tr\Theta)\geq\frac{n\malpha(\Theta)+n-1}{(m-1)\malpha(\Theta)+m}\,.
    \end{equation}
  \end{corollary}

  \begin{proof}
    Let $\gamma$ be a positive real number, $\gamma<\malpha(\Theta)$. By the hypothesis, for each $t$ large enough, there is a pair $\vec x\in\Z^m\backslash\{\vec 0\}$, $\vec y\in\Z^n$, such that
    \begin{equation*} 
      \Pi'(\vec x)\leq t,\qquad\Pi(\Theta\vec x-\vec y)\leq t^{-\gamma}.
    \end{equation*}
    For each such a $t$ by Theorem \ref{t:multitrans} there is a pair $\vec x'\in\Z^m$, $\vec y'\in\Z^n\backslash\{\vec 0\}$, such that
    \begin{equation} \label{eq:primes_inequalities_2}
      \Pi'(\vec y')\leq\Delta_d^{-\frac{1}{d-1}}t^{\frac{(m-1)\gamma+m}{d-1}},\quad
      \Pi(\tr\Theta\vec y'-\vec x')\leq\Delta_d^{-\frac{1}{d-1}}t^{-\frac{n\gamma+(n-1)}{d-1}}.
    \end{equation}
    With $s=\Delta_d^{-\frac{1}{d-1}}t^{\frac{(m-1)\gamma+m}{d-1}}$ the inequalities \eqref{eq:primes_inequalities_2} can be rewritten as
    \begin{equation*} 
      \Pi'(\vec x')\leq s,\qquad\Pi(\Theta\vec x'-\vec y')\leq s^{-\gamma'},
    \end{equation*}
    where
    \begin{equation*} 
      \gamma'=\frac{n\gamma+(n-1)+\varkappa(t)}{(m-1)\gamma+m-\varkappa(t)}\,,\quad\varkappa(t)=\frac{\ln\Delta_d}{\ln t}\,.
    \end{equation*}
    Taking into account that $s$ continuously depends on $t$ and $s\to\infty$ as $t\to\infty$, we get
    \[ \malpha(\tr\Theta)\geq\sup_{\gamma<\malpha(\Theta)}\gamma'=\frac{n\malpha(\Theta)+(n-1)}{(m-1)\malpha(\Theta)+m}\,. \]
  \end{proof}

  Modifying the proof of Corollary \ref{cor:tr_beta_and_mbeta} in a similar way, one can easily get the following statements:

  \begin{corollary} \label{cor:tr_alpha_and_malpha}
    If $n=1$, then
    \begin{equation*} 
      \alpha(\tr\Theta)\geq\frac{\malpha(\Theta)-m}{(m-1)\malpha(\Theta)+m}\,.
    \end{equation*}
  \end{corollary}

  \begin{corollary} \label{cor:alpha_and_malpha}
    If $m=1$, then
    \begin{equation*} 
      \malpha(\Theta)\geq\alpha(\Theta)\geq\frac{n\malpha(\Theta)-1}{n(n-1)\malpha(\Theta)+n^2-n+1}\,.
    \end{equation*}
  \end{corollary}

  It is quite natural to expect that \eqref{eq:dysonlike_malpha} can be improved, as in the case of ordinary Diophantine exponents we have the following statement (proved in \cite{german}):

  \begin{theorem} \label{t:my_inequalities}
    For all positive integers $n$, $m$, not equal simultaneously to $1$, 
    we have
    \begin{equation*} 
      \alpha(\tr\Theta)\geq
      \begin{cases}
        \dfrac{n-1}{m-\alpha(\Theta)}\,,\quad\ \ \text{ if }\ \alpha(\Theta)\leq1, \\
        \dfrac{n-\alpha(\Theta)^{-1\vphantom{\big|}}}{m-1}\,,\quad\text{ if }\ \alpha(\Theta)\geq1.
      \end{cases}
    \end{equation*}
  \end{theorem}

  Unfortunately, it is not clear, whether a similar statement holds for multiplicative exponents. The method used in \cite{german} to prove Theorem \ref{t:my_inequalities} fails in the multiplicative case because of the non-convexity of the star body determined by the inequality $\Pi(\vec x)\leq1$.

  Neither is it clear, whether there is a multiplicative analogue of the remarkable relation
  \[ \alpha(\Theta)^{-1}+\alpha(\tr\Theta)=1 \]
  proved by Jarn\'{i}k \cite{jarnik_tiflis} in the case $n=1$, $m=2$.

\vskip 10mm

\noindent
Oleg N. {\sc German} \\
Moscow Lomonosov State University \\
Vorobiovy Gory, GSP--1 \\
119991 Moscow, RUSSIA \\
\emph{E-mail}: {\fontfamily{cmtt}\selectfont german@mech.math.msu.su, german.oleg@gmail.com}

\end{document}